\documentclass[a4paper,11pt]{amsart}
\usepackage{graphicx} 

\usepackage[colorlinks,citecolor=red,urlcolor=blue,bookmarks=false,hypertexnames=true]{hyperref}
\usepackage{amsfonts}
\usepackage{amsmath}
\usepackage{dsfont}
\usepackage{amssymb}
\usepackage{mathrsfs}
\usepackage{amsthm}
\usepackage{nicefrac,eufrak}
\usepackage{bbm}
\usepackage{enumerate}
\usepackage{url,bm}
\usepackage{capt-of}
\usepackage{graphicx}
\usepackage{algorithm}
\usepackage{comment}
\usepackage{imakeidx}
\usepackage[dvipsnames]{xcolor}
\usepackage{soul}
\usepackage{cleveref}

\newtheorem*{mainthm}{Theorem A}
\newcommand{\R}{\mathbb{R}}
\newcommand{\N}{\mathbb{N}}
\newcommand{\C}{\mathbb{C}}
\newcommand{\Z}{\mathbb{Z}}

\newcommand{\calC}{\mathcal{C}}

\newcommand{\la}{\lambda}

\newcommand{\wh}[1]{\widehat{#1}}

\newcommand{\Aut}{\operatorname{Aut}}

\newtheorem{theorem}{Theorem}[section]
\newtheorem{lemma}[theorem]{Lemma}
\newtheorem{corollary}[theorem]{Corollary}
\newtheorem{proposition}[theorem]{Proposition}

\newtheorem*{problem}{Problem}

\newtheorem{example}[theorem]{Example}
\newtheorem*{question}{Question}

\newtheorem{definition}[theorem]{Definition}

\newcommand{\calX}{\mathcal{X}}

\def\R{\mathbb{R}}
\def\Z{\mathbb{Z}}
\def\N{\mathbb{N}}
\def\C{\mathbb{C}}
\def \supp{\mathrm{supp}}

\def \dG{\, {\rm d}m_{G}}

\keywords{Phase retrieval, LCA groups, completeness, uniqueness sets, short-time Fourier transform, Paley-Wiener spaces.}
\subjclass[2020]{43A70, 42C30, 94A20}

\begin{document}

\title{Phase retrieval from short-time Fourier transform in LCA groups}

\author[N. Accommazzo]{Natalia Accomazzo}
\author[D. Carando]{Daniel Carando}
\author[R. Nores]{Rocío Nores}
\author[V. Paternostro]{Victoria Paternostro}
\author[S. Velazquez]{Sebastián Velazquez}

\address{Natalia Accomazzo\\
Departamento de Matem\'atica
\\Facultad de Ciencias Exactas y Naturales\\ Universidad de Buenos Aires
\\ Pabellón I, Ciudad Universitaria \\
(1428) Buenos Aires,   Argentina.}
\email{naccomazzo@dm.uba.ar}

\address{Daniel Carando\\
Departamento de Matem\'atica
\\Facultad de Ciencias Exactas y Naturales\\ Universidad de Buenos Aires
and IMAS (UBA--CONICET)
\\ Pabellón I, Ciudad Universitaria \\
(1428) Buenos Aires,   Argentina.}
\email{dcarando@dm.uba.ar}

\address{Rocío Nores\\
Departamento de Matem\'atica
\\Facultad de Ciencias Exactas y Naturales\\ Universidad de Buenos Aires
and IMAS (UBA--CONICET)
\\ Pabellón I, Ciudad Universitaria \\
(1428) Buenos Aires,   Argentina.}
\email{rnores@dm.uba.ar}

\address{Victoria Paternostro\\
Departamento de Matem\'atica
\\Facultad de Ciencias Exactas y Naturales\\ Universidad de Buenos Aires
and IMAS (UBA--CONICET)
\\ Pabellón I, Ciudad Universitaria \\
(1428) Buenos Aires,   Argentina.}
\email{vpater@dm.uba.ar}

\address{Sebastián Velazquez \\
Department of Mathematics\\ King's College London\\ Strand\\
London WC2R 2LS\\ United Kingdom.}
\email{sebastian.velazquez@kcl.ac.uk}

\begin{abstract}
    We study the short-time Fourier transform phase retrieval problem in locally compact abelian groups. Using probabilistic methods, we show that for a large class of groups $G$ and compact subsets $K\subseteq G$ there exists a window function and a uniformly discrete set in $G\times \wh{G}$ allowing phase retrieval in $L^2(K)$. We also study the obstructions for STFT phase retrieval on $L^2(G)$, motivating the restriction to compactly supported function spaces.
\end{abstract}

\maketitle

\section{Introduction}
The phase retrieval problem arises naturally in different applications, such as crystallography and signal processing. Regardless of its various implementations, this topic has become a very active area of research, receiving significant attention from diverse areas like complex analysis, harmonic analysis and representation theory.

This paper is dedicated to the study of the short-time Fourier transform (STFT) phase retrieval problem in locally compact abelian groups. Readers who are interested in its applications to ptychography and other related areas are referred to \cite{ptychography,FaSt20, GrKoRa20}. For a locally compact abelian group $G$, we will write $\wh{G}$ for its dual group, which is also locally compact and abelian.
The short time Fourier transform with respect to a window function $g$ is the operator $V_g: L^2(G)\rightarrow L^2(G\times \wh{G})$ satisfying
$$ V_gf(x,\xi)= \int_G f(t)\overline{g(t-x)}\,\overline{ \langle t,\xi\rangle} \dG(t) =\langle f, M_\xi T_xg\rangle_{L^2(G)} $$
for every $(x,\xi)\in G\times\wh{G}$. When dealing with the real-world applications of this theory, one is restricted by the fact that measuring devices are often only able to take phaseless samples. In our formulation, this translates to the problem of recovering a function $f$
from the values $|V_gf(S)|$ for some (computable) set $S\subseteq G\times \wh{G}$. Of course, the best possible outcome is to be able to identify $f$ up to multiplication by scalars of modulus one (this is, a \emph{global phase}). The hypothesis of $f$ being compactly supported is often harmless, although quite convenient for tackling the issue. With this in mind, our STFT phase retrieval problem can be stated as follows:

\begin{problem}
    Let $G$ be a locally compact abelian group and $K\subseteq G$ be a compact set. Find a window function $g$, together with small sets $\Lambda\subseteq G$ and $\Gamma\subseteq \wh{G}$ such that every $f\in L^2(K)$ is  uniquely determined by $|V_gf(\Lambda\times \Gamma)|$ up to a global phase. When this holds, it is said the triple $(g,\Lambda, \Gamma)$ does phase retrieval in $L^2(K)$.
\end{problem}

This problem has received a considerable amount of attention, although almost exclusively in the case where $G=\R^d$ \cite{PR4,PR3,PR2,  PR1}. In \cite{Grohsetal}, it is shown that a triple $(g,\Lambda, \Gamma)$ does phase retrieval in $L^2(K)$ for a compact $K$ provided that some systems of translations related to $g$ are complete in the set $\calC(K)$ of continuous functions supported in $K$, and $\Gamma$ is a uniqueness set for $PW_K(\R^d)$. In particular, they are able to provide window functions $g$ that do phase retrieval when paired with lattices $\Gamma,\Lambda\subseteq \R^d$ (see Theorem \ref{thm:phaseret} and Theorem \ref{window Rd} below).

To our best knowledge, it is still unknown whether STFT phase retrieval is possible in more general groups. We give a positive answer for a large class of groups and compact sets. Our methods are of probabilistic nature and are inspired by the connection with the completeness properties of translates of window functions outlined in \cite{Grohsetal}. Before stating our results, let us first briefly discuss the structure of locally compact abelian (LCA) groups. By \cite[Theorem 4.2.1]{DeitmarEchterhoff}, every LCA group is of the form
$$G\simeq \R^d\times G_0$$
for $d\geq0$ and a LCA group $G_0$ having an open compact subgroup $H$ (here $\simeq$ means a  topological and algebraic  isomorphism.). Observe that if $G_0$ is not compact then of course $H$ is a proper subgroup. A particularly nice class of groups of this kind are \emph{dually expansible} groups. These are defined in Definition \ref{defexp} and have been an active subject of research in the last decades, see for instance \cite{BeBe04,BE04}.
A subset $S\subseteq G$ is said to be \emph{uniformly discrete} if there exists an open neighborhood $U\ni 0$ such that $(s+U) \cap S=\{s\}$ for every $s\in S$.  We are now ready to state our main contribution to the problem:

\begin{mainthm} \label{thm-main}
    Let $G\simeq \R^d \times G_1\times\dots\times G_n$ be a locally compact abelian group, where $d\geq 0$ and each $G_i$ is not compact and contains a proper open compact subgroup $H_i$. Let $K\subseteq G$ be a compact set. Suppose that for every $1\leq i \leq n$ we have either
    \begin{enumerate}
        \item[(1)] $G_i$ is dually expansible or discrete, or
        \item[(2)] the $i$-th projection of $K$ on $G_i$ is contained in $H_i$.
    \end{enumerate}
    Then there exists  $g\in \calC(G)$ and a uniformly discrete set $\Lambda\times \Gamma\subseteq G\times \wh{G}$ such that every $f\in L^2(K)$ is uniquely determined by $|V_gf(\Lambda\times \Gamma)|$ up to a global phase.
\end{mainthm}

Additionally, we establish that, when extra structure is imposed on $\Gamma$, phase retrieval cannot be achieved on the entire space $L^2(G)$ for general LCA groups, motivating the restriction to compactly supported functions in Theorem A.
More precisely, we show that if $\Gamma \subseteq \widehat{G}$ is such that $\overline{\langle \Gamma \rangle} \subsetneq \widehat{G}$, then for any window function $g\in L^2(G)$ there exist functions $f_1$ and $f_2$ with $|V_g f_1(G\times\Gamma)| = |V_g f_2(G\times\Gamma)|$, while $f_1$ and $f_2$ are not related by a global phase. This was motivated by \cite{PR1,PR4}, where the authors tackle this issue in the case $G=\mathbb R^d$ and obtain analogous results when $\Gamma$ is a lattice or a set derived from a lattice.

The paper is organized as follows. In Section \ref{sec:premilinaries}, we introduce some preliminary concepts and notation that will be used throughout the paper. Section \ref{sec:main-tools} serves a dual purpose. On one hand, we construct uniqueness sets for Paley-Wiener spaces. On the other hand, using probabilistic methods, we design functions and uniformly discrete sets such that the system of translates of these functions along the sets is complete. Both topics are essential for proving Theorem A. In Section \ref{sec:proofs}, we combine the results from Section \ref{sec:main-tools} to provide a  proof of Theorem A. Finally, in Section \ref{sec:nonunique} we prove that whenever $\Gamma \subseteq \widehat{G}$ satisfies $\overline{\langle \Gamma \rangle} \subsetneq \widehat{G}$, no choice of $g\in L^2(G)$ allows $(g,G, \Gamma)$ to perform  phase retrieval in $L^2(G)$.

\section{Preliminaries}\label{sec:premilinaries}

In this section we will review some basic definitions and facts on the topics treated in this article. We refer to \cite{BE04,DeitmarEchterhoff,Rudin} for further details and complete proofs.

We will understand a \emph{locally compact abelian group} $G$ - \emph{LCA group} for short - to be an abelian group $(G,+)$ which is also a locally compact topological Hausdorff space such that both multiplication and inversion are homeomorphisms. The Haar measure on $G$ will be denoted by $m_G$. Along the rest of the article, every group homomorphism between topological groups is assumed to be continuous. We now state the so-called  \emph{first structure theorem} for LCA groups.

\begin{theorem}\label{structure}\cite[Theorem 4.2.1]{DeitmarEchterhoff}
    Let $G$ be an LCA group. Then $G\simeq \R^d\times G_0$, where $d$ is a non-negative integer and $G_0$ is an LCA group containing an open compact subgroup.
  \end{theorem}

Given an LCA group $G$, we denote by $\wh G$ its Pontryagin dual. Since the dual of $\wh{G}$ is naturally isomorphic to the original group, for  $\xi\in\wh G$ and $x\in G$ we write $\langle x,\xi\rangle $ to indicate either the character $\xi$ applied to $x$ (i.e. $\xi(x)$) or the character  $x$ applied to $\xi$. For a subgroup  $H\subseteq G$, its \emph{annihilator} is denoted by $H^\perp$ and is defined as
$H^\perp=\{ \xi\in\wh G:\, \langle h,\xi\rangle=1 \,\forall \,h\in H\}$.

For a closed subgroup $H\subseteq G$, the dual group  of the quotient group $G/H$, $\widehat{G/H}$,  is  isomorphic to $H^\perp$, and $\wh H$ is isomorphic to $\wh G/H^\perp$.
When $H$ is an open and compact subgroup, then so is
$H^\perp\subseteq\wh G$. As a consequence,   both quotients $G/H$ and
$\wh G/H^\perp$ are discrete.

A \emph{section} of the quotient group $G/H$ is a \emph{measurable} set of representatives. This is, a measurable set which contains exactly one element of each coset. Notice that if $C\subseteq G$ is a section for the quotient $G/H$ and  $x\in C$, we have $\{x\}=C\cap (H+x)$. In particular, if $H$ is an open subgroup, then $C$ is a discrete subset of $G$.

For a fixed $A\in \Aut(G)$, we consider a new measure $\mu_A$ defined by
$ \mu_A(U)=m_G(AU)$ where $U$ is a Borel set of $G$. This application is a non-zero Haar measure on $G$. Therefore, there is a unique positive number $|A|$, the so-called {\it modulus of $A$}, such that $\mu_A=|A|m_G$.
The \emph{adjoint operator} $A^*:\wh{G}\to \wh{G}$ of $A$ is the automorphism satisfying $\langle Ax,\xi\rangle=\langle x, A^*\xi\rangle$ for all $x\in G$ and $\xi\in\wh G$. Straightforward computation shows that this assignment satisfies $|A^*|=|A|$ and $(A^nH)^\perp=A^{*-n}H^\perp$ for every subgroup $H$ of $G$ and every $n\in\mathbb Z$.

\begin{definition}\cite[Definition 2.5]{BeBe04} \label{defexp}
    Let $G$ be an LCA group, $H\subseteq G$ an open compact subgroup and $A\in \Aut(G)$. We say that $A$ is \emph{expansive} with respect to $H$ if the following hold true:
    \begin{enumerate}
        \item[(1)] $H\subsetneq AH$;
        \item[(2)] $\bigcap_{n\leq0} A^nH =\{0\}$.
    \end{enumerate}
    When $H$ is fixed or clear from the context, we simply say that $A$ is expansive. A group is said to be \emph{expansible} if it admits an expansive automorphism. If instead $\wh{G}$ is an expansible group then we  say that $G$ is a \emph{dually expansible} group.
\end{definition}

Expansiveness can also be characterized by the action of the adjoint automorphism, as the next lemma shows.

\begin{lemma}\cite[Lemma 2.6]{BeBe04} Let $G$ be an LCA group with an open compact subgroup $H\subseteq G$, and let $A\in \Aut(G)$.
Then $H\subseteq AH$ (resp. $H\subsetneq AH$) if and only if $H^{\perp}\subseteq A^*H^{\perp}$ (resp. $H^{\perp}\subsetneq A^*H^{\perp}$).
Moreover, if $H\subseteq AH$ then
\begin{align} \label{exp}\bigcup_{n\geq0}A^nH=G\iff \bigcap_{n\leq0}A^{*n}H^{\perp}=\{0\}.\end{align}
\end{lemma}

{In particular, by the above lemma, in the case where $H\subsetneq AH$  the group $G$ is dually expansive if and only if $G=\bigcup_{n\geq0}A^nH$.}

For $f\in L^1(G)$, its \emph{Fourier transform} is defined as
$$\widehat{f}(\xi):=\int_Gf(x)\overline{\langle x,\xi\rangle}\dG(x), \quad \xi\in \widehat{G}.
$$
As usual, this operator can be extended to an isomorphism from $L^2(G)$ onto $L^2(\widehat{G})$.
For an LCA group $G$ with an open and compact subgroup $H$, we consider the
following normalization of the Haar measures involved: we fix $m_G$ such that $m_G(H)=1$ and $m_{\wh G}$ such that $m_{\wh G}(H^\perp)=1$. This choice guarantees that the Fourier transform between $L^2(G)$ and $L^2(\wh G)$ is an isometry (see \cite[Comment (31.1)]{HR70}).

\begin{definition}
    Let $G$ be an LCA group and $\Omega \subseteq \widehat{G}$ a Borel set. The \emph{Paley-Wiener space associated to} $\Omega$ is
    $$PW_\Omega(G):=\{f \in L^2(G): \supp(\widehat{f})\subseteq \Omega\}.$$
\end{definition}

    It is worth pointing out that when $\Omega$ is a Borel set of finite measure, every element of $PW_\Omega$ is also a continuous function in $G$. Indeed, if $f\in PW_\Omega$, then we have $\widehat{f} \in  L^2(\Omega) \subset L^1(\Omega)$. Since $\widehat{f}$ is integrable, $f$ must be continuous.

\begin{definition}
    Let $G$ be an LCA group and $\calX$ a set of functions defined on $G$. A set  $\Upsilon\subseteq G $ is a uniqueness set for $\calX$   if
    $$f\in \calX\textrm{ and }f(\upsilon)=0\,\,\textrm{for all }\upsilon\in \Upsilon\,\, \Longrightarrow f=0.$$
\end{definition}

\noindent We will require our sample sets to be separated in the following sense.

\begin{definition}
    Let $G$ be a topological group. A subset $\Lambda$ is said to be \emph{uniformly discrete} if there exists a neighbourhood $U\subseteq G$ of the identity such that every translate of $U$ contains at most one element of $\Lambda$.
\end{definition}

Finally, for every $x\in G$,  the {\it translation operator} by $x$ of a function $f\in L^2(G)$ is given by
$T_xf(y)=f(y-x)$ for $m_G$-a.e. $y\in G$. For $\xi\in\wh G$, the {\it modulation operator} by $\xi$ of a function $f\in L^2(G)$ is defined by
$M_\xi f(y)=\langle y,\xi\rangle f(y)$ for $m_G$-a.e. $y\in G$. With this, we can recall the definition of one of the main concepts treated in this work.

\begin{definition}
 Fixed a window function $g\in L^2(G)$, the  \emph{short-time Fourier transform} (STFT) of $f\in L^2(G)$ is the function
$$V_gf(x,\xi)=\langle f, M_\xi T_xg\rangle_{L^2(G)} \,\,\quad \mbox{for every}\,\,\quad (x,\xi)\in G\times\wh{G}.$$
\end{definition}

\section{Main tools: uniqueness sets and completeness of translates}\label{sec:main-tools}

By Theorem \ref{structure}, we know that every LCA group is isomorphic to $\R^d\times G_0$ where $G_0$ contains an open compact subgroup. In order to prove our main result, note first that the findings in \cite{Grohsetal} show that Theorem A  holds for $G=\R^d$.
To show that it holds as it is presented,  we begin by stating the  following generalisation of  \cite[Theorem 2.1]{Grohsetal} to LCA groups. The proof appearing in that work translates mutatis mutandis to this setting, using Proposition \ref{prop:completeness->inyect} below and \cite[Section 2.2]{FeiGroI} instead of the references therein. For $s\in G$ and $g:G\to \C$, we denote by $g_s$ the function given by $g_s(y)=\overline{g}(y)T_sg(y)$, for $y\in G$.

\begin{theorem}\label{thm:phaseret} 
	Let $G$ be an LCA group, $\Lambda \subset G$, $K\subset G$ a compact set, and $g\in L^{\infty}_{loc}(G)$ such that, for every  $s\in K-K$, the operator
	\begin{align}\label{eq: injective op}
	    C(g,s):L^1(K)\to \C^{\Lambda}, \qquad h \mapsto \Big(\int_K h(t)\overline{T_x g_s(t)}\dG(t)\Big)_{x\in \Lambda}\end{align}
	is injective.
Let  $\Gamma\subset \widehat{G}$ is a uniqueness set for  $PW_{K-K}(\widehat{G})$. Then every  $f\in L^2(K)$ is uniquely determined, up to a factor of absolute value 1, by $|V_gf(\Lambda\times \Gamma)|$.
\end{theorem}

The hypothesis on the injectivity of the operators in Theorem \ref{thm:phaseret} is difficult to tackle. Inspired in \cite{Grohsetal}, we will therefore replace it with a slightly stronger condition, namely the completeness of some systems of translates. This notion turns out to be better adapted for our purposes and has been widely studied along the literature (e.g. \cite{Lan72, Liehr25, Zalik78}).

\begin{definition}
    Let $K\subseteq G$ be a compact set. We will say that a system $\{f_i\}_{i\in I}\subseteq \calC(G)$ \emph{is complete in} $\calC(K)$ if $\{f_i\vert_K\}_{i\in I}\subseteq \calC(K)$ is complete respect to $\|\cdot\|_{\infty}$.
\end{definition}

We  now make the connection of completeness and the injectivity of the aforementioned operators explicit. 

\begin{proposition} \label{prop:completeness->inyect}
    If $F=\{f_i\}_{i\in I}\subseteq \calC (K)$ is complete, then the operator
    \begin{align*}
        C_F:  L^1&(K)\to \C^{I}\\
         & h \longmapsto \Big(\int_K h(t)\overline{f_i(t)}\dG(t)\Big)_{i\in I}
    \end{align*}
    is injective.
\end{proposition}
\begin{proof} Let $h\in L^1(K)$ be such that $C_F(h)=0$. Being $F$ complete, we have
$$ \int_K h(t) \overline{g(t)} \dG(t)=0$$
for every $g\in \calC(K)$. Since  $h\in L^1(K)$, the map $L^\infty(K)\to \C$ defined by
$$ g\mapsto \int_K h(t) \overline{g(t)} \dG(t)$$
is weak-star continuous. Since $\calC(K)$ is weak-star dense in $L^\infty (K)$ we can conclude that $h=0$. Thus, $C_F$ is injective.
\end{proof}

As the previous results suggest, uniqueness sets and completeness of translations are crucial to show our main results.  As they are also of independent interest, we analyze them separately.

\subsection{Uniqueness sets for Paley-Wiener spaces.}

Our next goal is to show, under certain assumptions, the existence of uniformly discrete uniqueness sets for Paley–Wiener spaces. Since these results are of independent interest, we first state them for an arbitrary LCA group and appropriate sets. We then specialize to the case of a dual group and a compact set of the form 
$K-K$, which leads to the theorem below. For further references on the study of uniqueness in Paley–Wiener spaces, see \cite{aku1, lai2021conjugate, zhang2024gabor, Thakur}.

\begin{theorem}\label{cor-uniqueness}
    If $G$ is compact, discrete or dually expansible, then for every compact set $K\subset G$  there exists a uniformly discrete uniqueness set $\Gamma$ for $PW_{K-K}(\wh G)$.
\end{theorem}

The proof follows by combining a series of results that we will prove next. 
Let $G$ be an LCA group with  $H\subseteq G$ an open compact subgroup and fix $C\subseteq G$ a section for
the quotient $G/H$. The next proposition was shown in \cite{BE04}:
\begin{proposition} \cite[Proposition 2.1]{BE04}
    Let $f\in L^2(G)$, $A\in Aut (G)$ and $n\in\Z$. Then $\supp\,\wh f \subset (A^*)^nH^{\perp}$ if and only if $f$ is constant on every (open) coset of $A^{-n}H$ in $G$.
\end{proposition}
As an immediate consequence, one has that the  functions in $PW_{H^\perp}(G)$ are precisely those that are constant on all the cosets of $H$ in $G$. Moreover, we can straightforwardly derive  the following decomposition for the elements in the Paley-Wiener space.
\begin{corollary}   \label{coro:sampling}
    Every $f\in PW_{H^{\perp}}(G)$ is determined by the values it takes in any section $C$ of $G/H$. More precisely,
    \[f(x)= \sum_{c\in C} f(c)\chi_{H+c}(x), \qquad \text{for }f\in PW_{H^{\perp}}(G).\]
\end{corollary}

\noindent We now take  $L\subset G$ to  be a compact subgroup. For $\xi \in \wh{G}$ we have
   $$ \wh{\chi_L}(\xi)=\int_{G} \chi_L(t) \overline{\langle t,\xi\rangle} \,dm_G(t) = \int_L \langle t, \xi|_L\rangle \dG(t).$$
   If $\xi\in L^{\perp}$, then $\int_L \langle t, \xi|_L\rangle \dG(t)=m_{G}(L)$. On the other hand, if $\xi\notin L^{\perp}$, then
there exist some $t_0\in L$ such that $\langle t_0, \xi\rangle\neq 0$. Since the Haar measure is invariant under translation we have
$$\int_L \langle t, \xi|_L\rangle \dG(t) = \int_L \langle t-t_0+t_0, \xi|_L\rangle \dG(t)= \langle t_0, \xi|_L\rangle \int_L \langle t, \xi|_L\rangle \dG(t).$$
The last integral is finite because $L$ is compact, so the previous equality shows  $\int_L \langle t, \xi|_L\rangle \,dm_G(t)=0$.
This altogether gives that
\begin{equation}\label{eq:charH}
    \wh{\chi_L}=m_{G}(L)\chi_{L^{\perp}}
\end{equation}
for any compact subgroup $L\subset G$

By Corollary \ref{coro:sampling}, every set that contains a section of $G/H$ is a uniqueness set for $PW_{H^{\perp}}(G)$. We will now show that this condition is also necessary.

\begin{proposition} Let \label{prop:supp} be
  $G$ an LCA group and $H\subseteq G$ a compact and open subgroup. Then   $\Upsilon\subseteq G $ is a uniqueness set for $PW_{H^\perp}(G)$ if and only if $\Upsilon$ contains a section of $G/H$.
\end{proposition}

\begin{proof}
    First notice that for every $x\in G$ the function given by $f(y)=T_x \chi_H(y)=\chi_{H+x}(y)$ satisfies $\wh{f}=\overline{\langle x,\cdot\rangle} \wh{\chi_H}$ and then, by \eqref{eq:charH}, it
    belongs to $PW_{H^\perp}$.

    Now assume that $\Upsilon\subseteq G $ is a uniqueness set for $PW_{H^\perp}(G)$ and suppose that $\Upsilon$ does not contain any section of $G/H$. Then, there  exists some $x_0\in G$ such that $\Upsilon\cap (H+x_0)=\emptyset$. If we consider $f=\chi_{H+x_0}$, we have that $f\in PW_{H^\perp}$ and  $f(\Upsilon)=0$. However, $f\neq0$  which is a contradiction. As already noted, the other implication is immediate from Corollary \ref{coro:sampling}.
\end{proof}

\begin{proposition}
    Let $\Upsilon\subseteq G$ and $\Omega\subseteq \wh{G}$ be a Borel set. Then $\Upsilon$ is a uniqueness set for $PW_{\Omega}(G)$ if and only if it is a uniqueness set for $PW_{\Omega+\xi}(G)$ for every $\xi\in \wh{G}$.
\end{proposition}
\begin{proof}
   Let $\xi\in\wh{G}$. It is immediate to see that $\supp\,( \wh{f})\subseteq \Omega$ if and only if $\supp\, (T_{\xi}\wh{f})\subseteq \Omega+\xi$. Thus $f\in PW_{\Omega}(G)$ if and only if $\langle \cdot, \xi\rangle f\in PW_{\Omega+\xi}(G)$ and therefore they have the same uniqueness sets.
   \end{proof}

\begin{proposition}\label{prop:uniq-compact}
    Let $G$ be a compact abelian group. Given a compact subset $N\subseteq\wh{G}$ there exist a finite uniqueness set $\Upsilon\subset G$ for $PW_N(G)$.
\end{proposition}
\begin{proof}
    Since $G$ is a compact group, its dual is a discrete group and so every compact subset $N$ of $\wh{G}$ has finitely many elements. On the other hand, $\{\langle\cdot,\eta\rangle, \eta \in \wh{G}\}$ is an orthonormal basis of $L^2(G)$ and then, given $f\in L^2(G)$, we can write it as
  \begin{equation}\label{eq:characters-bon}
       f(\cdot)=\sum_{\eta\in\wh{G}}\wh{f}(\eta) \langle\cdot,\eta\rangle.
  \end{equation}
Since $f\in PW_N(G)$ if and only if $\wh{f}(\eta)=0$ for all $\eta\notin N$, by \eqref{eq:characters-bon} we have that $PW_N(G)=span\{\eta_i\}_{i=0}^{n}$ where $N=\{\eta_0,\cdots,\eta_n\}$.

In order to construct $\Upsilon$ we  introduce a useful notation. For every $x\in G$ let $X(x)$ be the element of $\C^{n+1}$ given by $X(x)=(\langle x,\eta_0\rangle,\langle x,\eta_1\rangle,\cdots,\langle x,\eta_n\rangle)$.

Now, fix any $x_0\in G$ and call $X^0:=X(x_0)$.
If for each $x\in G$ there exists $\alpha_x\in\C$ such that $X(x) = \alpha_x X^0$, then for $f\in PW_N(G)$ we have
   \begin{align*}
       f(x)=\sum_{j=0}^n \wh{f}(\eta_j) \langle x,\eta_j\rangle &= \alpha_x \sum_{j=0}^n \wh{f}(\eta_j)  \langle x_0,\eta_j\rangle= \alpha_x f(x_0).
       \end{align*}
    Thus,  $\Upsilon=\{x_0\}$ is a uniqueness set for $PW_N(G)$.

    If not, there exists $x_1\in G\setminus\{x_0\}$ such that the set $\{X^0,X^1\}\subset\C^{n+1}$ is linearly independent, where $X^1:=X(x_1)$. As before, if for every $x\in G$ exist $\alpha_{x}^0, \alpha_{x}^1\in\C$ such that $X(x)=\alpha_{x}^0X^0+\alpha_{x}^1X^1$ then for every $f\in PW_N(G)$ we have
    $$f(x)=\alpha_{x}^0f(x_0)+\alpha_{x}^1f(x_1)$$
    and hence $\Upsilon=\{x_0,x_1\}$ is a uniqueness set for $PW_N(G)$. If not, there exist $x_2\in G\setminus\{x_0,x_1\}$ such that the set $\{X^0, X^1, X^2\}$ is linearly independent, where $X^2:=X(x_2)$. If $\Upsilon=\{x_0,x_1, x_2\}$ is a uniqueness set we stop, and if not we continue with  the same procedure.
    Note that at most, we have to iterate the process  $n$ times to get a uniqueness set for $PW_N(G)$.~\end{proof}

We are ready to prove the main result of this subsection.
\begin{proof}[Proof of  Theorem \ref{cor-uniqueness}]
If $G$ is compact, then $\wh{G}$ is a discrete group and, for every compact set $K\subset G$, it suffices to take $\Gamma=\wh{G}$ as uniqueness set for $PW_{K-K}(\wh{G})$. If $G$ is discrete, then $\wh{G}$ is a compact group. Hence it suffices to apply Proposition \ref{prop:uniq-compact} to the group $\wh G$ and the compact set $N=K-K$, for any compact set $K\subseteq G$.

If $G$ is dually expansive then $G=\cup_{n\geq0}A^nH$. Therefore, for every compact set $K\subset G$ there exist some $n_0\in\N$ such that $K-K\subseteq A^{n_0}H$. Since $
A^nH=[(A^*)^{-n}H^{\perp}]^{\perp}$ for every $n\in \mathbb N$, we can use  Proposition \ref{prop:supp} applied to $\wh{G}$ and $(A^*)^{-n_0}H^{\perp}$ to get that if
we take $\Gamma$  as a  section of $\wh{G}/(A^*)^{-n_0}H^{\perp}$, then $\Gamma$ is a uniqueness set for $PW_{A^{n_0}H}(\wh{G})$. As $PW_{K-K}(\wh{G})\subseteq PW_{A^{n_0}H}(\wh{G})$ is immediate that $\Gamma$ is also a uniqueness set for $PW_{K-K}(\wh{G})$.
\end{proof}

\subsection{Completeness of systems of  translates}

In this section we turn our efforts to find a function $g\in L^{\infty}_{loc}(G)$ and a set $\Lambda\subset G$ such that for every compact set $K\subset G$ the set of translates $\{T_xg_s\}_{x\in\Lambda}$ is complete in $\calC(K)$, where $g_s=\overline{g}T_sg$. We will study three base cases separately, namely when $G$ is $\R^d$, is discrete or contains a proper open compact subgroup, respectively.
In the first of these situations, the existence of both a set $\Lambda$ and a function $g$ satisfying the desired properties is guaranteed by the following statement (see the proof of \cite[Proposition 2.2]{Grohsetal}).
\begin{proposition}
     \label{window Rd}
    Let $g\in\calC(\R^d)$ be a multivariate complex-valued Gaussian
    $$g(x)=e^{-x^TAx},$$
    where $A\in\C^{d\times d}$ is a positive definite Hermitian matrix. If $K\subseteq\R^d$ is a compact set and $\Lambda$ is an arbitrary lattice, then $\{T_xg_s\}_{x\in\Lambda}$ is complete in $(\calC(K),\|\cdot\|_{\infty})$ for each $s\in K-K$.
\end{proposition}

The other two base cases are considered in the next subsections.

\subsubsection*{Completeness of translates in discrete groups} \addcontentsline{toc}{subsubsection}{Completeness in discrete groups}

We  now address the case where $G$ is a discrete group equipped with the counting measure. Our goal is to construct a  random continuous function $g:G\to\mathbb C$ having complete systems of translates $\{T_xg_s\}_{x\in G}$ with probability $1$.
We  now address the case where $G$ is a discrete group equipped with the counting measure. Our goal is to prove the following  deterministic  result by constructing a  random continuous function $g:G\to\mathbb C$ for which the system of translates $\{T_xg_s\}_{x\in G}$ is complete with probability $1$.

\begin{theorem}
 \label{prop:complete-discrete}
    Let $G$ be a discrete abelian group. There exists a continuous function $g$ such that  for every compact $K\subset G$ the system $\{T_xg_s\}_{x\in G}$ is complete in $\calC(K)$ for every $s\in K-K$.
\end{theorem}

\begin{proof} Since $G$ is discrete and second countable we can write $G=\{y_k\}_{k\in \mathbb N}$. Let $(\gamma_k)_{k\in \N}$ be a sequence of independent  standard Gaussian variables and define $g:G\to \mathbb C$ as $g(y_k)=\gamma_k$ for each $k\in\N$. We aim to show that, almost surely,  $g$ satisfies our requirements for every $K$.

First observe that for $s\in G$ and $j\in\N$ we have
$$
    g_s(y_k)= g(y_k)\overline{g(y_k-s)} \quad\text{and}\quad T_{y_j}g_s(y_k)=g(y_k-y_j)\overline{g(y_k-s-y_j)}.
$$
We claim  that for  each  $N\in \N$, there are indices  $j_1^N, \dots j_N^N$ such that for all $s\in \{y_1,\dots, y_N \}$, the $N\times N$ matrix
$$ A^s_N= \big( T_{y_{j^N_\ell}}g_s(y_k) \big)_{1\le k,\ell\le N} = \big( g(y_k-y_{j^N_\ell})\overline{g(y_k-s-y_{j^N_\ell})} \big)_{1\le k,\ell\le N} $$
has independent entries.  Indeed, we can choose $j_\ell^N$ inductively so that $y_k-y_{j^N_\ell}$ and $y_k-s-y_{j^N_\ell}$ are  different from $y_k-y_{j^N_m}$ and $y_k-s-y_{j^N_m}$ for $k=1,\dots,N$, $s\in \{y_1,\dots,y_N\}\setminus\{ 0\}$ and $m=1,\dots,\ell-1$. This can be done because we have infinitely many $j$'s to choose from.
Then, the claim follows because measurable functions of independent random variables are independent.

It is well-known that  a matrix  is almost surely invertible whenever its entries are absolutely continuous. Since this is the case for  each $ A^s_N$, we have that it is invertible with probability one. Moreover, as  we have countably many matrices $ A^s_N$ (with $N\in \N$ and $s\in \{y_1,\dots, y_N \}$), we conclude that, with probability one, every $ A^s_N$ is invertible. Thus, we can take a realisation of $(\Gamma_k)_{k\in \N}$ such that every $ A^s_N$ is invertible.

Now, given a compact (hence finite) set $K\subset G$, we choose $N$ such that $K\subset\{y_1,\dots,y_N\}$ and $K-K\subset\{y_1,\dots,y_N\}$. Then,  take any  $h\in \calC(K)\subset \calC(\{y_1,\dots,y_N\})$. Since $A^s_N$ is invertible, the system
$$ \sum_{k=1}^N  g(y_k-y_j)\overline{g(y_k-s-y_j)}\  \alpha_k = h(y_j)  \quad  j=1,\dots,N$$ has a (unique) solution.  This means that we can write $h=\sum_{k=1}^N  \alpha_k \  T_{y_j}g_s$, which shows that  $\{T_xg_s\}_{x\in G}$ is complete in $\calC(K)$ for every  $s\in K-K$.
\end{proof}

\subsubsection*{Completeness in the presence of a (proper) open compact subgroup
}\addcontentsline{toc}{subsubsection}{Completeness in dually expansible groups}

We now  focus on the last base  case mentioned before, where we are in the presence of a non-compact group having a proper subgroup which is both open and compact. In this situation we have the following result.

\begin{theorem}\label{thm-compl}
       Let $G$ be a non-compact LCA group with an open compact proper subgroup $H$. Then there exists  $g\in \calC(G)$ and a uniformly discrete set $\Lambda\subset G$ such that for every compact $K\subseteq H$, the system $\{T_xg_s\}_{x\in \Lambda}$ is complete in $\calC(K)$ for every $s\in K-K$.
Moreover, if $G$ is dually expansible, then the conclusion remains true  for \emph{every} compact $K\subset G$.
\end{theorem}

For the proof it is enough to   assume that $K=H$. Indeed, if $K\subseteq H$ is a compact set and  the  system $\{T_{x}g_s\}_{x\in\Lambda}$ is complete in $\calC(H)$, then so is in $\calC(K)$.

Now, for $\{x_k\}_{k\in\N}$ a  fixed    section of $G/H$ (which is non-trivial since $H$ is proper)
we will  construct a continuous random function $g\in\calC(G)$ such that, with probability one,  the system $\{T_{x_k}g_s\}_{k\in\N}$ is complete in $\calC(H)$ for every $s\in H$. For that purpose we fix $Z$ a Steinhaus random variable, i.e. a random variable uniformly distributed on $\mathbb T$, and $\{a_\mu\}_{\mu \in \wh{H}}$ a sequence of positive numbers satisfying $a_0^2> \sum_{\mu\neq 0} a_\mu^2$. For every $k\in\N$ and each $\mu\in\wh{H}$ we also define independent random variables $\lambda_{\mu,k}$ having the same distribution as $a_\mu Z$. In particular,
$$\mathbb{E}(\lambda_{\mu,k})=\mathbb{E}((\lambda_{\mu,k})^2)=0 \,\mbox{ and }\, Var(\lambda_{\mu,k})=a_\mu^2.$$
Observe that for a fixed $\mu\in\wh{H}$ the sequence  $\{\lambda_{\mu,k}\}_{k\in\N}$ is identically distributed.

Let us consider the random function $g\in \calC(G)$ defined by
\begin{equation}\label{eq:randomg}
    g= \sum_{\eta\in\wh{H}, k\in\N} \lambda_{\eta,k} \,\, T_{-x_k}\eta \,\,\chi_{H-x_k}.
\end{equation}

Straightforward calculation shows that for $s\in H$,
$$ T_{x_k}g_s\vert_H= \sum_{\eta\in \wh{H}} \Big( \sum_{\mu\in \wh{H}} \lambda_{\mu,k} \,\overline{\lambda_{\mu-\eta,k}} \,\,\mu(-s) \Big) \eta. $$
In order to understand the expected behaviour of these translates, we first need  some technical lemmas.

\begin{lemma}\label{lemma:averages} Let $\mu,\eta,\eta_0\in \wh{H}$ with $\eta_0\neq 0$. Then, with probability $1$ we have
    $$ \frac{1}{N}\sum_{k=1}^N \lambda_{\mu,k} \overline{\lambda_{\mu-\eta,k}} \,\,\overline{\lambda_{0,k}} \lambda_{-\eta_0,k}\xrightarrow{N\to\infty} \begin{cases}
         a_0^2a_{-\eta_0}^2 & \mbox{ if } \mu=0 \mbox{ and } \eta=\eta_0 \\
         0 & \mbox{ otherwise. }
    \end{cases}  $$
\end{lemma}
\begin{proof} By the law of large numbers,
    $$ \frac{1}{N}\sum_{k=1}^N \lambda_{\mu,k} \overline{\lambda_{\mu-\eta,k}} \,\,\overline{\lambda_{0,k}} \lambda_{-\eta_0,k}\xrightarrow{N\to\infty} \mathbb{E}(\lambda_{\mu,1} \overline{\lambda_{\mu-\eta,1}} \,\,\overline{\lambda_{0,1}} \lambda_{-\eta_0,1}), $$
almost surely. Since these variables are independent with zero expected value, then
$$  \mathbb{E}(\lambda_{\mu,1} \overline{\lambda_{\mu-\eta,1}} \,\,\overline{\lambda_{0,1}} \lambda_{-\eta_0,1})=0$$
if every subscript appears only once. Moreover, since $\eta_0\neq 0$ and $\mathbb{E}((\lambda_{\mu,1})^2)=0$ this expression also vanishes unless $\mu=0$ and $\eta=\eta_0$, in which case we obtain $ \frac{1}{N}\sum_{k=1}^N \lambda_{\mu,k} \overline{\lambda_{\mu-\eta,k}} \,\,\overline{\lambda_{0,k}} \lambda_{-\eta_0,k}=a_0^2a_{-\eta_0}^2$.
\end{proof}
\noindent The same reasoning leads to the following:

\begin{lemma} \label{lemma:averages0}
    Let $\mu,\eta\in \wh{H}$. Then, with probability $1$ we have
    $$ \frac{1}{N}\sum_{k=1}^N \lambda_{\mu,k} \overline{\lambda_{\mu-\eta,k}}\xrightarrow{N\to\infty}  \begin{cases}
         a_\mu^2 & \mbox{ if }  \eta=0 \\
         0 & \mbox{ otherwise. }
    \end{cases} $$
\end{lemma}

\begin{proposition}\label{prop:completenessbyprobability}
Let $\{x_k\}_{k\in\N}$ be a section of $G/H$ and $g$ the random function defined in \eqref{eq:randomg}. Then with probability one, the system $\{T_{x_k}g_s\}_{k\in \N}$ is complete in $\calC(H)$ for all $s\in H$.
\end{proposition}
\begin{proof}
    In order to keep a clear notation, let us write $Y_k^s= T_{x_k}g_s\vert_H $. Recall that since $H$ is compact, the algebra of its  characters is dense in $\calC(H)$. Then it suffices to show that every character belongs to the closure of ${\rm span}(Y_k^s)_{k\in \N}$. Let $\eta_0$ be a non-trivial character on $H$. We will now show that with probability $1$, for every $s\in H$ the expression $ \frac{1}{N}\sum_{k=1}^N Y_k^s \overline{\lambda_{0,k}} \lambda_{-\eta_0,k} $ uniformly converges to a non-zero scalar multiple of $\eta_0$ when $N\to\infty$.  For this, we take a realisation of the random variables $\lambda_{\mu,k}$ such that the countably many convergences given in Lemmas \ref{lemma:averages} and \ref{lemma:averages0} hold.

For every $\varepsilon>0$, let $A\subseteq \wh{H}$ be a finite set satisfying $0\in A$ and
    \begin{equation}
        \label{eqA} \sum_{\mu\notin A} a_\mu < \varepsilon.
    \end{equation}
Now take a finite set $B\subseteq\wh{H}$ such that $\eta_0\in B$ and
    \begin{equation}
        \label{eqB}
         \sum_{\eta \notin B} a_{\mu-\eta}<\varepsilon \,\, \mbox{ for every } \, \mu\in A.
    \end{equation}
    Lemma \ref{lemma:averages} implies
    $\frac{1}{N}\sum_{k=1}^N \lambda_{\mu,k} \overline{\lambda_{\mu-\eta,k}} \,\,\overline{\lambda_{0,k}} \lambda_{-\eta_0,k}\xrightarrow{N\to\infty} \delta_{(\eta_0,0)}(\eta,\mu)a_0^2a^2_{-\eta_0}.$
    Since both $A$ and $B$ are finite, there exists $N_1$ large enough such that
    \begin{equation} \label{eqAB}
        \Big|\frac{1}{N}\sum_{k=1}^N \lambda_{\mu,k} \overline{\lambda_{\mu-\eta,k}} \,\,\overline{\lambda_{0,k}} \lambda_{-\eta_0,k}-\delta_{(\eta_0,0)}(\eta,\mu)a^2_0a^2_{-\eta_0}\Big|<\frac{\varepsilon}{\# A \# B}
    \end{equation}
    for every $\eta\in B$ and $\mu \in A$ provided that $N\geq N_1$.
Given $t\in H$ and using the notation
\begin{align*}
\Big| \frac{1}{N}\sum_{k=1}^N Y_k^s (t)\overline{\lambda_{0,k}} \lambda_{-\eta_0,k} - a_0^2 a_{-\eta_0}^2\eta_0 (t) \Big|  = \star,
\end{align*}
an easy rearrangement gets us
\begin{align*}
\star &= \Big|\sum_{\eta\in \wh{H}}\Big(\sum_{\mu \in \wh{H}}\mu(-s)\frac1N\sum_{k=1}^N\lambda_{\mu,k}\overline{\lambda_{\mu-\eta,k}}\,\,\overline{\lambda_{0,k}}\lambda_{-\eta_0,k}\Big)\eta(t)-a_0^2a^2_{-\eta_0}\eta_0(t)\Big|\\
&=\Big|\sum_{\eta\in \wh{H}}\Big(\sum_{\mu \in \wh{H}}\mu(-s)\frac1N\sum_{k=1}^N\lambda_{\mu,k}\overline{\lambda_{\mu-\eta,k}}\,\,\overline{\lambda_{0,k}}\lambda_{-\eta_0,k}-\delta_{(\eta_0,0)}(\eta,\mu)a_0^2a_{-\eta_0}^2\Big)\eta(t)\Big| \\
&\le \sum_{\eta\in \wh{H}}\sum_{\mu \in \wh{H}}\Big|\frac1N\sum_{k=1}^N\lambda_{\mu,k}\overline{\lambda_{\mu-\eta,k}}\,\,\overline{\lambda_{0,k}}\lambda_{-\eta_0,k}-\delta_{(\eta_0,0)}(\eta,\mu)a_0^2a_{-\eta_0}^2\Big| .
\end{align*}
We can now make use of our choices of $A$ and $B$ and write the above expression in the form
\begin{align*}
\sum_{\eta\in B}\sum_{\mu \in A } \Big|\frac1N\sum_{k=1}^N&\lambda_{\mu,k}\overline{\lambda_{\mu-\eta,k}}\,\,\overline{\lambda_{0,k}}\lambda_{-\eta_0,k}-\delta_{(\eta_0,0)}(\eta,\mu)a_0^2a_{-\eta_0}^2\Big|  + \sum_{\eta\notin B}\sum_{\mu \in A}\Big|\frac1N\sum_{k=1}^N\lambda_{\mu,k}\overline{\lambda_{\mu-\eta,k}}\,\,\overline{\lambda_{0,k}}\lambda_{-\eta_0,k}\Big|\\
&+\sum_{\eta \in \wh{H}}\sum_{\mu \notin A}\Big|\frac1N\sum_{k=1}^N\lambda_{\mu,k}\overline{\lambda_{\mu-\eta,k}}\,\,\overline{\lambda_{0,k}}\lambda_{-\eta_0,k}\Big|,
\end{align*}
where we are using that the function $\delta_{(\eta_0,0)}(\eta,\mu)a_0^2a_{-\eta_0}^2$ vanishes along $B^c\times A$ and $\wh{H}\times A^c$. By ~\eqref{eqAB}, we see that the first term is bounded by $\varepsilon$ whenever $N\geq N_1$. Since $|\lambda_{\mu,k}|=a_\mu$ for every $\mu$ and $k$, we can safely assure that
\begin{align*}
    \star \le \varepsilon + \sum_{\eta\notin B}\sum_{\mu \in A}a_{\mu}a_{\mu-\eta}a_{0}a_{-\eta_0}+ \sum_{\eta\in \wh{H}}\sum_{\mu \notin A}a_{\mu}a_{\mu-\eta}a_{0}a_{-\eta_0}.
\end{align*}
For the last term, we can use \eqref{eqA} and show
$$ \sum_{\eta\in \wh{H}} \sum_{\mu\notin A} a_\mu a_{\mu-\eta} \,\, a_0 a_{-\eta_0} = a_0 a_{-\eta_0} \sum_{\mu\notin A}a_\mu  \sum_{\eta\in \wh{H}} a_{\mu-\eta} = a_0 a_{-\eta_0}  \sum_{\mu\notin A}a_\mu  \sum_{\nu\in \wh{H}}a_\nu < C \varepsilon  .$$
Similarly, \eqref{eqB} implies
$\sum_{\eta\notin B } \sum_{\mu\in A}  a_\mu a_{\mu-\eta} \,\, a_0 a_{-\eta_0} < C\varepsilon$,
which shows that $\eta_0\in \overline{{\rm span}}(Y_k^s)_{k\in \N}$ whenever $\eta_0$ is not trivial. For the remaining case, we will show that
$$\frac{1}{N} \sum_{k=1}^N Y_k^s \xrightarrow{N\to \infty} \sum_{\mu \in \wh{H}} a_\mu^2 \mu(-s),$$
uniformly.
Indeed, from Lemma \ref{lemma:averages0} we know that
\[\frac{1}{N} \sum_{k=1}^N \lambda_{\mu,k} \overline{\lambda_{\mu-\eta,k}} \xrightarrow{N\to\infty}\delta_{0}(\eta)a_{\mu}^2.\]
We define the sets $A$ and $B$ and the constant $N_1$ as we did before, with the minor change that we only ask for $0$ to belong to $B$. With this setup, we can bound
\begin{align*}
    \Big|\frac1N&\sum_{k=1}^N Y^s_k(t)-\sum_{\mu\in \wh{H}}a^2_{\mu}\mu(-s)\Big| = \Big|\sum_{\eta\in\wh{H}}\Big[\sum_{\mu\in\wh{H}}\mu(-s)\Big(\frac1N\sum_{k=1}^N\lambda_{\mu}\overline{\lambda_{\mu-\eta,k}}-\delta_0(\eta)a_\mu^2\Big)\Big]\eta(t)\Big|\\
    &\le \sum_{\eta\in B}\sum_{\mu\in A}\Big|\frac1N\sum_{k=1}^N\lambda_{\mu,k}\overline{\lambda_{\mu-\eta,k}}-\delta_0(\eta)a_{\mu}^2\Big| + \sum_{\eta\notin B}\sum_{\mu\in A} |\lambda_{\mu,k} \overline{\lambda_{\mu-\eta,k}}| + \sum_{\eta\in\wh{H}}\sum_{\mu\notin A} |\lambda_{\mu,k} \overline{\lambda_{\mu-\eta,k}}|.
\end{align*}
    From here we can proceed as in the prior case to show that the constant function $\sum_{\mu\in \wh{H}}a^2_{\mu}\mu(-s)$  belongs to the subspace spanned by $(Y_k^s)_{k\in \N}$. Since $a_0^2> \sum_{\mu\neq 0} a_\mu^2$, 
     this constant is different from 0 and then the trivial character is also in that subspace, which ends the proof.  \qedhere
\end{proof}

\begin{proof}[Proof of Theorem \ref{thm-compl}]
The first part of Theorem \ref{thm-compl} is a simple consequence of Proposition \ref{prop:completenessbyprobability}. If $G$ is a dually expansive group, for every compact set $K\subset G$ there exists some $n_0\in\N$ such that both $K$ and $K-K$ are contained in $A^{n_0}H$. So we can construct $g$ as in \eqref{eq:randomg} by taking $A^{n_0}H$ as an open compact subgroup of $G$ and $\Lambda$ a section of $G/A^{n_0}H$. The result follows applying Proposition  \ref{prop:completenessbyprobability}.
\end{proof}

\section{Proof of Theorem A}\label{sec:proofs}
At this point we can combine Theorem \ref{thm:phaseret} with the results of the previous sections to obtain Theorem A. The following lemmas show how to deal with products of simpler groups.

\begin{lemma}\label{lemma:productcompleteness}
For $i=1,2$ let $G_i$ be an LCA group, $K_i\subseteq G_i$ a compact set and $\Lambda_i\subseteq G_i$. Suppose further that $g_i\in \calC(G_i)$ satisfies that the system $\{T_{x}g_i \}_{x\in \Lambda_i}$ is complete in $\calC(K_i)$. Then $\{T_{(x,y)}g_1g_2\}_{x\in\Lambda_1,y\in\Lambda_2}$ is complete in $\calC(K_1\times K_2)$.
\end{lemma}
\begin{proof}
First notice that it suffices to prove that $\calC(K_1)\times\calC(K_2)\subseteq\overline{{\rm span}}\{T_{(x,y)}g_1g_1\}_{x\in\Lambda_1,y\in\Lambda_2}$ since the algebra of separate variable functions is dense in $\calC(K_1\times K_2)$ by the Stone–Weierstrass theorem.
Moreover, from the equality $T_{(x,y)}g_1g_2=T_xg_1\, T_yg_2$ it follows that $\varphi_1\,\varphi_2\in {\rm span}\{T_{(x,y)}g_1g_2\}_{x\in\Lambda_1\, y\in\Lambda_2}$ for every $\varphi_1\in{\rm span}\{T_xg_1\}_{x\in\Lambda_1}$ and $\varphi_2\in{\rm span}\{T_yg_2\}_{y\in\Lambda_2}$.

Now, given $f=f_1\,f_2$ with $f_i\in\calC(K_i)$ for $i=1,2$ and $\varepsilon>0$ there exist $\varphi_1 \in {\rm span}\{T_xg_1\}_{x\in\Lambda_1}$ and $\varphi_2\in{\rm span}\{T_yg_2\}_{y\in\Lambda_2}$  such that $\|f_i-\varphi_i\|_{\calC(K_i)}<\varepsilon$ for $i=1,2$. But then
\begin{align*}
    \|f_1f_2-\varphi_1\varphi_2\|_{\calC(K_1\times K_2)}&\leq \|f_1f_2-f_1\varphi_2\|_{\calC(K_1\times K_2)}+\|f_1\varphi_2-\varphi_1\varphi_2\|_{\calC(K_1\times K_2)} \\
    &= \|f_1\|_{\calC(K_1)}\|f_2-\varphi_2\|_{\calC(K_2)} + \|\varphi_2\|_{\calC(K_2)} \|f_1-\varphi_1\|_{\calC(K_1)}\\
    &\leq \|f_1\|_{\calC(K_1)} \varepsilon + (\|f_2\|_{\calC(K_2)}+\varepsilon) \varepsilon \xrightarrow{\varepsilon\to0}0,
\end{align*}
and the lemma follows.
\end{proof}

\begin{lemma}\label{lem:productoPW}
For $i=1,2$ let $G_i$ be an LCA group and  $\Omega_i\subseteq \wh G_i$ a Borel set.
    Let $f\in PW_{\Omega_1\times \Omega_2}(G_1\times G_2) $. Then $f_x=f(x,\cdot)\in PW_{\Omega_2} (G_2)$ for a.e. $x\in G_1$
\end{lemma}
\begin{proof}
Fix $f\in PW_{\Omega_1\times \Omega_2}(G_1\times G_2)$ and   take $E\subseteq G_1$, $F\subseteq \wh{G}_2\setminus \Omega_2$  subsets of finite measure. Since $\wh{f}$ and $\wh{\chi_E}\chi_F$ have disjoint supports, we have
\begin{align*}
    \int_{G_1}\int_{\wh{G_2}} \wh{f_x}(\eta) \chi_E(x) \chi_F(\eta) \,{\rm{d}}m_{\wh G_2}(\eta) \,{\rm{d}}m_{G_1}(x) &= \int_{G_1}\int_{G_2} f_x(y) \chi_E(x) \chi_F^{\vee}(y)\,{\rm{d}}m_{G_2}(y)\,{\rm{d}}m_{G_1}(x) \\
    &=\int_{\wh{G_1}}\int_{\wh{G_2}}\wh{f}(\gamma, \eta) \wh{\chi_E}(\gamma) \chi_F(\eta) \,{\rm{d}}m_{\wh G_2}(\eta) \,{\rm{d}}m_{\wh G_1}(\gamma) =0.
\end{align*}
Thus $\wh{f_x}(\eta)=0$ for a.e. $(x,\eta)\in G_1\times \wh{G_2}\setminus \Omega_2$. In particular, the set $N=\{ (x,\eta)\in G_1\times \wh{G_2}\setminus \Omega_2 : \wh{f_x}(\eta)=0\}$ satisfies $m_{G_1\times\wh{G_2}}(N)=0$. Applying Fubini's Theorem we obtain $m_{\wh{G_2}}(N_x)=0$ for a.e. $x\in G_1$. This means that $\wh{f_x}=0$ a.e. in $\wh{G_2}\setminus \Omega_2$ for a.e. $x\in G_1$.
\end{proof}

\begin{lemma}\label{lemma:productuniqueness}
    If $\Gamma_i$ is a uniqueness set for $PW_{\Omega_i}(G_i)$ for $i=1,2$, then $\Gamma_1\times \Gamma_2$ is a uniqueness set for the space $PW_{\Omega_1\times \Omega_2}(G_1\times G_2)$.
\end{lemma}
\begin{proof}
Let $f\in PW_{\Omega_1\times \Omega_2}(G_1\times G_2)$ such that $f(\Gamma_1\times\Gamma_2)=0$.
For each $\gamma_1\in\Gamma_1$ we have that $f_{\gamma_1}(\Gamma_2)=0$, therefore, by Lemma \ref{lem:productoPW}, $f_{\gamma_1}=0$ in $G_2$. This shows $f(\Gamma_1\times G_2)=0$. Now fix some $y\in G_2$. Again, we are in the situation $f_y(\Gamma_1)=0$. Hence $f=0$ in $G_1\times G_2$.
\end{proof}

\noindent We are now ready to prove our main theorem.

\begin{proof}[Proof of Theorem A]
Let $G\simeq \R^d \times G_1\times\dots\times G_n$ be an LCA group, where each $G_i$ contains a proper open compact subgroup $H_i$. We use the notation $G_0$ for the first factor and $\pi_i$ for the $i$-th projection. Observe that any compact $K\subseteq G$ satisfies $K\subseteq \prod_{i=0}^n \pi_i(K)$. In particular, it is enough to consider compacts $K$ for the form $K=K_0\times \dots \times K_n$. 

Let $K_i\subset G_i$ be a compact set. If $G_i$ is discrete or dually expansive then by Theorem \ref{cor-uniqueness} there exists a uniformly discrete uniqueness set $\Gamma_i\subset\wh{G_i}$ for the space $PW_{K_i-K_i}(\wh{G_i})$. Also, by Theorem \ref{prop:complete-discrete} and Theorem \ref{thm-compl} there exist a function $g_i\in\calC(G_i)$ and a uniformly discrete set $\Lambda_i\subset G_i$ such that the system $\{T_x(g_i)_s\}_{x\in\Lambda_i}$ is complete in $\calC(G_i)$ for every $s\in K_i-K_i$.
If on the other hand $K_i\subseteq H_i$, then we can take $\Gamma_i$ to be a section of $\wh{G_i}/H_i^\perp$ as in Proposition \ref{prop:supp}. In this case, the existence of $\Lambda_i$ and $g_i$ is assured by Theorem \ref{thm-compl}. For the case $i=0$ we refer to the constructions made in \cite{Grohsetal}.

By Lemma \ref{lemma:productcompleteness}, the function $g:=g_0\dots g_n$ and the uniformly discrete set $\Lambda:= \Lambda_0\times\dots\times\Lambda_n$ are such that the system $\{ T_xg_s\}_{x\in \Lambda}$ is complete in $\calC(G)$ for every $s\in K-K$. Using Proposition \ref{prop:completeness->inyect} we know that the application $C(g,s)$ given in \eqref{eq: injective op} is injective.
Also, by Lemma \ref{lemma:productuniqueness} we can guarantee that $\Gamma:= \Gamma_0\times\dots\times\Gamma_n$ is a uniqueness set for $PW_{K-K}(\wh{G})$.
Now we can conclude applying Theorem \ref{thm:phaseret}.
\end{proof}

\section{Limitations for phase retrieval in $L^2(G)$}\label{sec:nonunique}

Although the short-time Fourier transform provides a powerful tool for local time-frequency analysis, it is unclear whether STFT phase retrieval on the whole $L^2(G)$ is feasible. 
When $G=\mathbb R^d$ it was proved in \cite{PR1,PR4} that elements in $L^2(\mathbb R^d)$ are not determined (up to a global phase) from the absolute values of their STFT on a lattice (or a set that can be derived from a lattice), regardless of the choice of the window function. Using the same methods, we will now deduce the analogous result for an arbitrary LCA group $G$, motivating the restriction to the subspaces $L^2(K)$ for some compact set $K\subseteq G$, as we do in our Theorem A.

Let $g$ be a non-zero element in $L^2(G)$ and $\Lambda=\{\lambda_1,\dots,\lambda_m\}\subseteq G$ be a finite set. Let us consider the functions  $f=\sum_{i=1}^mc_i T_{\lambda_i}\mathcal Rg$ and $f_{\times}=\sum_{i=1}^m\overline{c_i} T_{\lambda_i}\mathcal Rg$, where $c_i\in\C$ and $\mathcal Rg(x)=g(-x)$. Unless we impose some tight restrictions on the coefficients $\{c_i\}$ we see that the functions $f$ and $f_{\times}$ are \textit{not} equal up to a global phase, as this next lemma shows.

\begin{lemma}\label{nonequivalent} In the situation above, $f$ and $f_\times $ are equal up to a global phase if and only if there exists some $\alpha\in \R$ such that $c_j\in e^{i\alpha}\R$ for every $1\leq j \leq m$.  
\end{lemma}
\begin{proof}
    This is straightforward and follows in the same manner as in the first part of \cite[Theorem 3.8]{PR1}, with minor modifications.
\end{proof}

On the other hand, this construction yields functions whose STFT with respect to $g$ have the same absolute value on a large subset of time–frequency shifts.

\begin{lemma}\label{phaseless measurements}
    Let $g\in L^2(G)$ and $\Lambda=\{\lambda_1,\dots,\lambda_m\}\subset G$. If $f=\sum_{i=1}^m c_i T_{\lambda_i}(\mathcal Rg)$ and $f_{\times}=\sum_{i=1}^m\overline{c_i} T_{\lambda_i}\mathcal Rg$ for some $c_i \in\C$, then $|V_gf(x,\xi)|=|V_gf_\times(x,\xi)|$ for all $(x,\xi)\in G\times\Lambda^{\perp}$.
\end{lemma}
\begin{proof}
We will follow the proof of \cite[Theorem 3.1]{PR1} closely. Let $f$ and $\Lambda$ be as above and $(x,\xi)\in G\times\Lambda^{\perp}$. Given that $V_gf(x,\xi)=\langle f, M_\xi T_x g\rangle=\sum_{i=1}^m c_i \langle T_{\la_i} \mathcal{R}g, M_\xi T_xg\rangle$, by standard properties of the time-shift translates we get

    \begin{align*}
      |V_gf(x,\xi)|&=\left|\sum_{i=1}^m c_i\langle \mathcal{R}g,T_{-\la_i} M_\xi T_{x}g\rangle\right|\\
       &=\left|\sum_{i=1}^m c_i \overline{\langle \la_i,\xi \rangle} \langle \mathcal{R}g, M_\xi T_{x-\la_i}g\rangle\right|.
    \end{align*}
Since $\xi\in\Lambda^{\perp}$, this equation reduces to $|V_gf(x,\xi)|=\left|\sum_{i=1}^m c_i\langle \mathcal{R}g, M_\xi T_{x-\la_i}g\rangle\right|$. In the case of $f_\times$, this yields 

    \begin{align} \label{ftimes}
        |V_gf_\times(x,\xi)|&=\left|\sum_{i=1}^m \overline{c_i} \langle \mathcal{R}g, M_\xi T_{x-\la_i}g\rangle\right|=\left|\sum_{i=1}^m c_i  \langle \overline{\mathcal{R}g, M_\xi T_{x-\la_i}g\rangle}\right|.
    \end{align}
\noindent
    Using the change of variable $-z=y-(x-\la_i)$ , these inner products take the form:
    
    \begin{align*}
        \overline{\langle \mathcal{R}g, M_\xi T_{x-\la_i}g\rangle}&= \int_G \overline{g(-y)} \langle y,\xi\rangle g(y-(x-\la_i))\dG (y) \\
        &= \int_G\overline{g(z-(x-\la_i))} \langle -z+(x-\la_i),\xi\rangle g(-z)\dG(z)\\
        &= \langle x-\la_i,\xi\rangle \int_G \overline{g(z-(x-\la_i))}\overline{\langle z,\xi\rangle} g(-z)\dG(z)\\
        &=\langle x,\xi\rangle \langle \mathcal{R}g,M_\xi T_{x-\la_i}g\rangle.
     \end{align*}
\noindent
Replacing this in \eqref{ftimes}, we get  

\begin{align*}
    |V_gf_\times(x,\xi)|&=\left|\sum_{i=1}^m c_i \langle x,\xi\rangle \langle \mathcal{R}g,M_\xi T_{x-\la_i}g\rangle\right|\\
    &= \left|\sum_{i=1}^m c_i \langle \mathcal{R}g,M_\xi T_{x-\la_i}g\rangle\right|=|V_gf(x,\xi)|,
\end{align*}
concluding the proof.
\end{proof}

\begin{theorem} \label{teo:non-uniq}
Let $G$ be an LCA group, and $g\in L^2(G)$ an arbitrary window function. 
Let $\Gamma \subseteq \widehat G$ be a subset contained in a proper closed subgroup of $\widehat G$. 
Then there exist $f_1,f_2\in L^2(G)$ such that 
\begin{align}\label{spectograms}
|V_g f_1(x,\xi)| \;=\; |V_g f_2(x,\xi)|,
\qquad \forall (x,\xi)\in G\times \Gamma,
\end{align}
but $f_1$ and $f_2$ are not equal up to a global phase.
\end{theorem}

\begin{proof}
Since the closed subgroup generated by $\Gamma$ is a proper subset of $\wh{G}$, we have $\Gamma^{\perp}\neq\{0\}$. Thus, we can choose a finite subset $\Lambda=\{\la_1,\dots,\la_m\}\subseteq\Gamma^{\perp}$ with at least two different elements.

 Let us fix a finite sequence $\{c_i\}_{i=1}^m\subseteq\C$, not contained in $e^{i\alpha}\R$ for any $\alpha\in\R$, and consider $f_1=\sum_{i=1}^mc_i T_{\la_i}(\mathcal Rg)$ and $f_2=(f_1)_{\times}$. Since $\Gamma\subseteq\Lambda^{\perp}$, Lemma \ref{phaseless measurements} ensures
$$|V_gf_1(x,\xi)|=|V_gf_2(x,\xi)| \hspace{.3cm} \mbox{for all}\hspace{.1cm} (x,\xi)\in G\times\Gamma.$$
Finally, the hypothesis on the $c_i$'s allows us to apply Lemma \ref{nonequivalent} to conclude.    
\end{proof}

\begin{example} For a prime number $p$ denote by $\mathbb F_p$ the field of order $p$. Let $G$ be the group associated to the additive structure of the field $\mathbb F_p((t))$ of formal Laurent series in variable $t$:
$$\mathbb F_p((t))=\left\{ \sum_{n\geq n_0} a_nt^n: \, n_0\in\Z\,\,\mathrm{and}\,\, a_n\in\mathbb F_p \right\}.$$
Being a local field, $G$ is a self-dual LCA group with respect to the topology induced by the $t$-adic valuation (see for instance \cite[Proposition 7.1]{ramakrishnan2013fourier} ). The set of formal power series
$$\mathbb F_p[[t]]=\left\{ \sum_{n\geq 0} a_nt^n: \, a_n\in\mathbb F_p \right\}$$
is a compact open subgroup, and the sets $t^n\mathbb F_p[[t]]$ determine a fundamental system of neighborhoods of $0$. On the other hand,
$$ D=\left\{\sum_{n=-n_0}^{-1}a_nt^n: n_0\in\N, a_n\in\mathbb F_p\right\}$$
is a lattice such that
$\mathbb F_p((t))=\mathbb F_p[[t]]\oplus D$ and therefore any uniformly discrete set $\Gamma\subseteq D$ will satisfy the hypothesis of Theorem \ref{teo:non-uniq}.  
\end{example}

    In the case where $G=\mathbb R^d$, our hypotheses on $\Gamma$ are strictly weaker than being a lattice: $\Gamma$ does not need to be discrete, and we are requiring the quotient $\mathbb R^d/ \overline{\langle \Gamma \rangle}$ to be non-trivial rather than compact.
A more general concept than that of lattices in LCA groups is given by uniformly discrete subgroups. Of course, a uniformly discrete subgroup $\Gamma\subseteq \widehat{G}$ will satisfy the hypothesis of Theorem \ref{teo:non-uniq} (unless $\Gamma=\widehat{G}$).
For uniformly discrete \emph{subsets}, on the other hand, the situation is more involved. 
It is not always the case that a uniformly discrete subset $\Gamma\subseteq \widehat{G}$ satisfies the strict inclusion $\overline{\langle \Gamma \rangle }\subsetneq \widehat{G}$. This fails, for instance, when $\Gamma$ is a section of the quotient $G/H$ for $G=\mathbb Q_p$ and $H=\mathbb Z_p$, as the reader can easily check. This naturally raises the following

\begin{question}
    When $\overline{\langle \Gamma \rangle }=\widehat G$, for $\Gamma$ a uniformly discrete set, is there a natural condition that guarantees the existence of a window function $g$ such that every element in $L^2(G)$ is determined (up to a global phase) by the values of its short-time Fourier transform on $G\times\Gamma$, or perhaps even a smaller set?  
\end{question}

It is worth pointing out that a similar question, in the particular case of $G=\mathbb R^d$, was formulated in \cite[Problem 4.2]{PR4}.

 \section*{Acknowledgements}

This paper was carried on within the project ``Matem\'aticas en el Cono Sur 2", a collaborative workshop focused on tackling unsolved mathematical problems through teamwork.  The initiative encourages participation from all genders, with research groups led by women to promote female leadership and visibility in the field.
We thank the organizers for their efforts, their support  and for the hospitality extended to us during the meeting held in Montevideo, Uruguay, in February 2024.

We would like to thank Felipe Marceca for helpful conversations, which inspired us to use probabilistic tools.

We would also like to thank the referee for his/her comments, which helped us to improve the exposition, and in particular for his/her question that gave rise to the content of Section 5. 

N. A. and D. C. were supported by  PIP 11220200102366CO (CONICET) and UBACyT 20020220300242BA (UBA).
R. N. and V. P. were supported  PICT 2019-03968 (ANPCyT) and  PIP 11220210100087 (CONICET). S. V. was supported by EPSRC - United Kingdom.

\end{document}